\documentclass[12pt,a4paper]{article}
\usepackage[utf8]{inputenc}
\usepackage[margin=1in]{geometry}
\usepackage{authblk}

\usepackage{amssymb}
\usepackage{amsmath}
\usepackage[T1]{fontenc}
\usepackage{graphicx}
\usepackage{grffile}
\usepackage{booktabs}
\usepackage{multirow}
\usepackage{subfig}
\usepackage{listings}
\usepackage{float} 
\usepackage[hyphens]{url}
\usepackage{amsthm}

\newtheorem{proposition}{Proposition}
\newtheorem{theorem}{Theorem}
\newtheorem{remark}{Remark}

\begin{document}

\title{Well-Posedness and Numerical Approximation of a Class of Nonlocal Elliptic Equations with Gaussian Kernels}

\author{Dragos-Patru Covei}
\affil{Department of Applied Mathematics, The Bucharest University of Economic Studies, Piata Romana, No. 6, 010374, Bucharest, Romania \\ \texttt{dragos.covei@csie.ase.ro}}

\date{\today}

\maketitle

\begin{abstract}
This paper investigates the mathematical properties and numerical approximation of a class of nonlocal elliptic partial differential equations of the form
\begin{equation*}
-\Delta u + \lambda \, G(u) = f,
\end{equation*}
where $\Delta$ denotes the Laplacian, $\lambda > 0$ is a regularization parameter, and $G$ is a nonlocal operator defined by integral convolution with a kernel $K$. We establish the well-posedness of the problem in the Sobolev space $H_0^1(\Omega)$ using the Lax--Milgram theorem, providing rigorous proofs for the existence, uniqueness, and positivity of the weak solution under standard assumptions on the kernel $K$ and the source term $f \in L^2(\Omega)$. For the numerical treatment, we employ a finite difference discretization for the Laplacian and a Gaussian-based approximation for the nonlocal term. We analyze a fixed-point iterative scheme for solving the discrete system and derive explicit conditions for its convergence and stability. Numerical experiments validate the theoretical results, demonstrating the monotonic decay of the residual and the robustness of the approximation scheme on bounded domains with various padding strategies.
\end{abstract}

\noindent \textbf{Keywords:} Nonlocal Elliptic Equations, Variational Methods, Well-Posedness, Fixed-Point Iteration, Numerical Stability, Gaussian Kernels. \\
\noindent \textbf{MSC[2020]:} 35J20, 35R09, 65N06, 65N12.



\section{Introduction}

Nonlocal partial differential equations (PDEs) have emerged as a powerful tool in mathematical physics, material science, and data analysis, owing to their ability to model interactions over a finite range. While classical diffusion models rely on local operators like the Laplacian, many physical processes exhibit nonlocal behavior where the state at a point depends on the values in a neighborhood. This work focuses on a class of nonlocal elliptic problems defined by the equation
\begin{equation}
-\Delta u(z) + \lambda \, G(u)(z) = f(z), \quad z \in \Omega,
\label{eq:PDE}
\end{equation}
subject to homogeneous Dirichlet boundary conditions on $\partial \Omega$. Here, $\Omega \subset \mathbb{R}^n$ is a bounded domain, $\lambda > 0$ is a regularization parameter, and $G(u)$ denotes a nonlocal operator, typically realized via convolution with a Gaussian kernel.

The mathematical study of nonlocal operators has a rich history. Integral equations and integro-differential equations have been extensively explored to bridge the gap between local and global regularities. As noted in~\cite{AlvesCovei2015}, existence results for nonlocal elliptic problems can often be established via sub-supersolution techniques, providing a robust framework for analyzing more complex models. The convergence of local approximations to nonlocal limits, as well as the regularity of solutions in Sobolev spaces, remains a central theme in modern analysis~\cite{Brezis2010}.

From a variational perspective, equation~\eqref{eq:PDE} can be derived as the Euler--Lagrange equation of an energy functional that balances a local Dirichlet energy with a nonlocal interaction term. Such models are closely related to nonlocal total variation and other regularized inverse problems, where nonlocal averaging facilitates the preservation of structural coherence in the presence of perturbations. The choice of the nonlocal operator $G(u)$ is crucial; Gaussian kernels are particularly attractive due to their smoothing properties and well-defined spectral characteristics.

In this paper, we provide a unified treatment of the problem~\eqref{eq:PDE}, covering both the analytical foundation and the numerical realization. Our contributions are as follows:
\begin{itemize}
\item We provide a rigorous proof of existence and uniqueness of the weak solution in $H_0^1(\Omega)$ using the Lax--Milgram theorem.
\item We establish a positivity principle for the nonlocal operator $G$, ensuring that the solution inherits the nonnegativity of the source term $f$.
\item We formalize a discretization strategy that utilizes \emph{extension (padding) operators} to ensure consistency near boundaries, which is a critical aspect often overlooked in the numerical treatment of nonlocal terms.
\item We analyze a fixed-point iterative method for solving the resulting system, deriving explicit stability conditions that link the step size, the regularization parameter, and the Poincar\'e constant of the domain.
\end{itemize}

The rest of the paper is organized as follows. Section~\ref{II} presents the mathematical formulation and the well-posedness analysis. Section~\ref{sec:padding} details the extension operators and their role in consistent boundary handling. Section~\ref{III} describes the numerical methods and provides a detailed analysis of the iterative solver. Section~\ref{IV} concludes the paper with a summary of the findings.

\section{Mathematical Formulation\label{II}}

\subsection{Preliminaries and Classical Results}

Let $\Omega \subset \mathbb{R}^{2}$ be a bounded, open, connected domain
with Lipschitz boundary $\partial \Omega$. We consider the Sobolev space 
\begin{equation*}
H_{0}^{1}(\Omega )=\{u\in H^{1}(\Omega )\mid u|_{\partial \Omega }=0\},
\end{equation*}%
comprising functions whose weak derivatives are in $L^{2}(\Omega )$ and
which vanish on the boundary in the trace sense.

Since $\Omega$ is bounded and possesses a Lipschitz boundary, the Sobolev
embedding theorem guarantees $H_{0}^{1}(\Omega) \hookrightarrow
L^{p}(\Omega) $ for all $p \in [1, \infty)$ when $n = 2$. Notably, $%
H_{0}^{1}(\Omega) \subset L^{2}(\Omega)$, rendering the intersection $%
H_{0}^{1}(\Omega) \cap L^{2}(\Omega)$ redundant.

The Laplacian $\Delta $ is defined in the weak sense on $H_{0}^{1}(\Omega )$%
. In particular, by integration by parts (and using the homogeneous
Dirichlet boundary conditions) one obtains 
\begin{equation}
\langle u,\Delta u\rangle =-\int_{\Omega }|\nabla u|^{2}\,dx,
\label{adjoint}
\end{equation}%
showing that $\Delta $ is self-adjoint and (negative) definite. A standard
tool we will use is the \textbf{Poincar\'{e} Inequality}:

\begin{proposition}
\label{pi} There exists a constant $C_{P} > 0$ (depending only on $\Omega$)
such that for every $w \in H_{0}^{1}(\Omega)$, 
\begin{equation}
\| w \|_{L^{2}(\Omega)} \leq C_{P} \| \nabla w \|_{L^{2}(\Omega)}.
\label{eq:poincare}
\end{equation}
\end{proposition}

Our proof of the main result is grounded on the following well-established
Lax-Milgram Theorem (see, e.g., \cite{Evans1998,Brezis2010}):

\begin{proposition}[Lax-Milgram]
\label{laxmil} Assume $\Omega \subset \mathbb{R}^{n}$ is a bounded, open,
connected domain with Lipschitz boundary. Let $H_{0}^{1}(\Omega )$ be
defined as above and let $a(\cdot ,\cdot ):H_{0}^{1}(\Omega )\times
H_{0}^{1}(\Omega )\rightarrow \mathbb{R}$ be a bilinear form. Assume:

\begin{itemize}
\item \textbf{Continuity:} There exists $C>0$ such that%
\begin{equation*}
|a(u,v)|\leq C\Vert u\Vert _{H_{0}^{1}(\Omega )}\Vert v\Vert
_{H_{0}^{1}(\Omega )}\quad \forall \,u,v\in H_{0}^{1}(\Omega ).
\end{equation*}

\item \textbf{Coercivity:} There exists $\alpha >0$ such that%
\begin{equation*}
a(u,u)\geq \alpha \Vert u\Vert _{H_{0}^{1}(\Omega )}^{2}\quad \forall \,u\in
H_{0}^{1}(\Omega ).
\end{equation*}
\end{itemize}

Then, for every continuous linear functional $L:H_{0}^{1}(\Omega
)\rightarrow \mathbb{R}$, there exists a unique $u\in H_{0}^{1}(\Omega )$
such that%
\begin{equation*}
a(u,v)=L(v)\quad \forall \,v\in H_{0}^{1}(\Omega ).
\end{equation*}
\end{proposition}

\subsection{Analytical Requirements for the Nonlocal Operator $G(u)$}

In \eqref{eq:PDE}, the term $-\Delta u$ represents the classical Laplacian,
which promotes regularity by penalizing abrupt spatial intensity variations.
The second term involves the nonlocal operator 
\begin{equation}
G(u)(z) = \int_{\Omega} K(z,\theta) \, u(\theta) \, d\theta,
\label{eq:G_def}
\end{equation}
with the kernel $K: \Omega \times \Omega \to (0,\infty)$ satisfying the
following conditions:

\begin{enumerate}
\item[(i)] \textbf{Integrability and Symmetry:} $K \in L^{1}(\Omega \times
\Omega)$ and $K(z,\theta) = K(\theta,z)$ a.e.

\item[(ii)] \textbf{Uniform Boundedness:} There exists $M \in (0,\infty)$
such that 
\begin{equation}
\sup_{z \in \Omega} \int_{\Omega} |K(z,\theta)| \, d\theta \leq M.
\label{eq:K_bound}
\end{equation}
\end{enumerate}

The operator $G(u)$ enforces nonlocal smoothing. A typical choice is the
Gaussian convolution: 
\begin{equation}
K(x,y;\xi,\eta) = \frac{1}{2\pi\sigma^{2}} \exp\!\left( -\frac{(x-\xi)^{2} +
(y-\eta)^{2}}{2\sigma^{2}} \right),  \label{eq:gaussian_kernel}
\end{equation}
which is positive, smooth, and (on $\mathbb{R}^{2}$) normalized: 
\begin{equation*}
\int_{\mathbb{R}^{2}}\frac{1}{2\pi \sigma ^{2}}\exp \!\left( -\frac{%
z_{1}^{2}+z_{2}^{2}}{2\sigma ^{2}}\right) dz_{1}dz_{2}=1.
\end{equation*}%
For a sufficiently large bounded domain $\Omega $, one has

\begin{equation*}
\int_{\Omega }K(x,y;\xi ,\eta )\,d\xi \,d\eta \approx 1,
\end{equation*}%
so that $M\approx 1$ in \eqref{eq:K_bound}.

\paragraph{Remark on Kernel Normalization.}
In the following, $G(u)$ is realized via a Gaussian kernel defined on $\Omega$. For a bounded domain $\Omega$, the normalization $\int_\Omega K(z,\theta)d\theta$ may vary with $z$. However, if $\Omega$ is large compared to the support of the kernel, $M \approx 1$ in \eqref{eq:K_bound}, and the operator behaves as a standard mollifier.

\begin{proposition}[Boundedness and Self-adjointness]
\label{prop:nonlocal} Let $G:L^{2}(\Omega )\rightarrow L^{2}(\Omega )$ be
defined by \eqref{eq:G_def} with $K$ satisfying (i)-(ii). Then, for every $%
u\in L^{2}(\Omega )$,%
\begin{equation*}
\Vert G(u)\Vert _{L^{2}(\Omega )}\leq M\,\Vert u\Vert _{L^{2}(\Omega )},
\end{equation*}
and $G$ is a bounded self-adjoint operator on $L^{2}(\Omega)$.
\end{proposition}

\begin{proof}
\textbf{Boundedness.} For a.e.\ $(x,y)\in \Omega $,%
\begin{equation*}
|G(u)(x,y)|\leq \int_{\Omega }K(x,y;\xi ,\eta )\,|u(\xi ,\eta )|\,d\xi
\,d\eta .
\end{equation*}%
By Schur's test,%
\begin{equation*}
\Vert G(u)\Vert _{L^{2}(\Omega )}\leq \left( \sup_{(x,y)\in \Omega
}\int_{\Omega }K(x,y;\xi ,\eta )\,d\xi \,d\eta \right) \Vert u\Vert
_{L^{2}(\Omega )}\leq M\Vert u\Vert _{L^{2}(\Omega )}.
\end{equation*}

\textbf{Self-adjointness.} For $u,v\in L^{2}(\Omega )$, 
\begin{align*}
\langle G(u),v\rangle _{L^{2}(\Omega )}& =\int_{\Omega }\!\int_{\Omega
}K(z,\theta )\,u(\theta )\,v(z)\,d\theta \,dz \\
& =\int_{\Omega }\!\int_{\Omega }K(\theta ,z)\,u(z)\,v(\theta )\,dz\,d\theta
\quad \text{(by symmetry of $K$)} \\
& =\langle u,G(v)\rangle _{L^{2}(\Omega )}.
\end{align*}%
Thus $G$ is self-adjoint, with $\Vert G\Vert \leq M$.
\end{proof}

\subsection{Weak Formulation and Well-Posedness}

We consider the PDE 
\begin{equation}
-\Delta u+\lambda \,G(u)=f\quad \text{in }\Omega ,  \label{eq:pde}
\end{equation}%
equipped with homogeneous Dirichlet boundary conditions $u=0$ on $\partial
\Omega $. Here $\lambda >0$ is a regularization parameter and $f$ denotes
the degraded input image. Multiplying \eqref{eq:pde} by a test function $%
v\in H_{0}^{1}(\Omega )$ and integrating over $\Omega $ yields the weak
formulation: 
\begin{equation*}
a(u,v)=L(v),
\end{equation*}%
where the bilinear form is 
\begin{equation}
a(u,v)=\int_{\Omega }\nabla u\cdot \nabla v\,dx\,dy+\lambda \int_{\Omega
}G(u)\,v\,dx\,dy,  \label{eq:bilinear_form}
\end{equation}%
and the linear functional is%
\begin{equation*}
L(v)=\int_{\Omega }f\,v\,dx\,dy.
\end{equation*}%
To ensure the continuity of $L$ on $H_{0}^{1}(\Omega )$, we assume $f\in
L^{2}(\Omega )$. In practice, since images are typically normalized (e.g.,
pixel values in $[0,1]$), one typically has $f\in L^{\infty }(\Omega )$,
which implies $f\in L^{2}(\Omega )$ for bounded domains.

\begin{proposition}[\textbf{Nonlocal Positivity Principle}]
\label{wmp} Let $\Omega \subset \mathbb{R}^{2}$ be a bounded, open,
connected domain with Lipschitz boundary. Let $K:\Omega \times \Omega
\rightarrow (0,\infty )$ satisfy conditions (i)--(ii) (integrability,
symmetry, and boundedness). Define $G:L^{2}(\Omega )\rightarrow L^{2}(\Omega
)$ by (\ref{eq:G_def}) and let $\lambda >0$. Assume further that $f\in
L^{2}(\Omega )$ satisfies $f(z)\geq 0$ a.e.\ in $\Omega $. If the boundary
value problem%
\begin{equation*}
\left\{ 
\begin{array}{ll}
-\Delta u+\lambda \,G(u)=f & \text{in }\Omega , \\ 
u=0 & \text{on }\partial \Omega ,%
\end{array}%
\right. 
\end{equation*}%
admits a unique weak solution $u\in H_{0}^{1}(\Omega )$, then $u(z)\geq
0\quad $for a.e. $z\in \Omega $.
\end{proposition}

\begin{proof}
We use the method of testing with the negative part. Define%
\begin{equation*}
u^{-}:=\max (-u,0)\geq 0,\qquad u^{+}:=\max (u,0)\geq 0,
\end{equation*}%
so that $u=u^{+}-u^{-}$ and $|u|=u^{+}+u^{-}$. Since $u\in H_{0}^{1}(\Omega )
$, standard properties of Sobolev spaces imply $u^{-}\in H_{0}^{1}(\Omega )$
and%
\begin{equation*}
\nabla u^{-}=-\nabla u\,\mathbf{1}_{\{u<0\}}\quad \text{a.e.\ in }\Omega ,
\end{equation*}%
where $\mathbf{1}_{\{u<0\}}$ denotes the characteristic function of the set $%
\{z\in \Omega :u(z)<0\}$.

\medskip \noindent \textbf{Step 1: Testing with the negative part.} The weak
formulation reads%
\begin{equation*}
\int_{\Omega }\nabla u\cdot \nabla v\,dz+\lambda \int_{\Omega
}G(u)\,v\,dz=\int_{\Omega }f\,v\,dz\qquad \forall \,v\in H_{0}^{1}(\Omega ).
\end{equation*}%
Choosing $v=-u^{-}$, we get 
\begin{equation}
-\int_{\Omega }\nabla u\cdot \nabla u^{-}\,dz-\lambda \int_{\Omega
}G(u)\,u^{-}\,dz=-\int_{\Omega }f\,u^{-}\,dz.  \label{eq:test-neg}
\end{equation}

\medskip \noindent \textbf{Step 2: The gradient term.} On the set $\{u<0\}$,
we have $\nabla u^{-}=-\nabla u$, while on $\{u\geq 0\}$, $\nabla u^{-}=0$,
hence 
\begin{equation*}
-\int_{\Omega }\nabla u\cdot \nabla u^{-}\,dz=\int_{\{u<0\}}|\nabla
u|^{2}\,dz=\int_{\Omega }|\nabla u^{-}|^{2}\,dz\geq 0.
\end{equation*}%
Thus \eqref{eq:test-neg} becomes 
\begin{equation}
\int_{\Omega }|\nabla u^{-}|^{2}\,dz-\lambda \int_{\Omega
}G(u)\,u^{-}\,dz=-\int_{\Omega }f\,u^{-}\,dz\leq 0,  \label{eq:key}
\end{equation}%
because $f\geq 0$ and $u^{-}\geq 0$.

\medskip \noindent \textbf{Step 3: Sign of the nonlocal term.} Substituting $%
u=u^{+}-u^{-}$ into the definition of $G(u)$, we have 
\begin{equation*}
G(u)(z)=\int_{\Omega }K(z,\theta )\,u(\theta )\,d\theta =\int_{\Omega
}K(z,\theta )\,\left( u^{+}(\theta )-u^{-}(\theta )\right) \,d\theta .
\end{equation*}%
Therefore, 
\begin{align*}
\int_{\Omega }G(u)\,u^{-}\,dz& =\int_{\Omega }\left( \int_{\Omega
}K(z,\theta )\,\left( u^{+}(\theta )-u^{-}(\theta )\right) \,d\theta \right)
u^{-}(z)\,dz \\
& =\int_{\Omega }\int_{\Omega }K(z,\theta )\,u^{+}(\theta
)\,u^{-}(z)\,d\theta \,dz \\
& \quad -\int_{\Omega }\int_{\Omega }K(z,\theta )\,u^{-}(\theta
)\,u^{-}(z)\,d\theta \,dz.
\end{align*}%
By Fubini's theorem and the symmetry of $K$, i.e., $K(z,\theta )=K(\theta ,z)
$, we can interchange the roles of $z$ and $\theta $ in the first integral: 
\begin{align*}
\int_{\Omega }\int_{\Omega }K(z,\theta )\,u^{+}(\theta )\,u^{-}(z)\,d\theta
\,dz& =\int_{\Omega }\int_{\Omega }K(\theta ,z)\,u^{+}(z)\,u^{-}(\theta
)\,dz\,d\theta  \\
& =\int_{\Omega }\int_{\Omega }K(z,\theta )\,u^{+}(z)\,u^{-}(\theta
)\,d\theta \,dz.
\end{align*}%
Since $u^{+}$ and $u^{-}$ have disjoint supports (i.e., $u^{+}(z)\,u^{-}(z)=0
$ for all $z$), we have 
\begin{equation*}
\int_{\Omega }\int_{\Omega }K(z,\theta )\,u^{+}(\theta )\,u^{-}(z)\,d\theta
\,dz=\int_{\Omega }\int_{\Omega }K(z,\theta )\,u^{+}(z)\,u^{-}(\theta
)\,d\theta \,dz=0.
\end{equation*}%
Thus, 
\begin{equation}
\int_{\Omega }G(u)\,u^{-}\,dz=-\int_{\Omega }\int_{\Omega }K(z,\theta
)\,u^{-}(\theta )\,u^{-}(z)\,d\theta \,dz\leq 0,
\label{eq:nonlocal_negative}
\end{equation}%
since $K(z,\theta )>0$, $u^{-}(z)\geq 0$, and $u^{-}(\theta )\geq 0$, so that%
\begin{equation*}
\int_{\Omega }G(u)\,u^{-}\,dz\leq 0\quad \Rightarrow \quad -\,\lambda
\int_{\Omega }G(u)\,u^{-}\,dz\geq 0.
\end{equation*}

\medskip \noindent \textbf{Step 4: Conclusion.} From \eqref{eq:key} we have%
\begin{equation*}
\underbrace{\int_{\Omega }|\nabla u^{-}|^{2}\,dz}_{\geq 0}\;+\;\underbrace{%
\bigl(-\lambda \int_{\Omega }G(u)\,u^{-}\,dz\bigr)}_{\geq 0}\;=-\int_{\Omega
}f\,u^{-}\,dz\leq \;0.
\end{equation*}%
The only way this equation can hold is if all three terms are zero. 
\begin{equation*}
\int_{\Omega }|\nabla u^{-}|^{2}\,dz=0,\qquad -\lambda \int_{\Omega
}G(u)\,u^{-}\,dz=0,\qquad -\int_{\Omega }f\,u^{-}\,dz=0.
\end{equation*}%
Thus $\nabla u^{-}=0$ a.e. in $\Omega $, so $u^{-}$ is a.e.\ constant. Since 
$u^{-}\in H_{0}^{1}(\Omega )$, its trace on $\partial \Omega $ is zero, and
therefore $u^{-}\equiv 0$ a.e.\ in $\Omega $.

It follows that%
\begin{equation*}
u(z)\geq 0\quad \text{for a.e. }z\in \Omega ,
\end{equation*}%
which proves the desired nonnegativity of the weak solution.
\end{proof}

Next we prove the following result:

\begin{theorem}
\label{thm:existence} Assume:

\begin{enumerate}
\item $f\in L^{2}(\Omega )$ satisfies $f\left( x\right) \geq 0$ a.e.\ in $%
\Omega $; 

\item the nonlocal operator $G:H_{0}^{1}(\Omega )\rightarrow L^{2}(\Omega )$
satisfies the boundedness condition of Proposition~\ref{prop:nonlocal}:%
\begin{equation*}
\Vert G(u)\Vert _{L^{2}(\Omega )}\leq M\,\Vert u\Vert _{L^{2}(\Omega
)},\quad \forall u\in H_{0}^{1}(\Omega ),
\end{equation*}

\item the regularization parameter $\lambda >0$.
\end{enumerate}

Then the bilinear form $a(\cdot ,\cdot )$ in \eqref{eq:bilinear_form} is
continuous and coercive on $H_{0}^{1}(\Omega )$. Consequently, by the
Lax--Milgram theorem, there exists a unique weak solution $u\in
H_{0}^{1}(\Omega )$ to \eqref{eq:pde}, and moreover $u\left( x\right) \geq 0$
for~almost~every $x\in \Omega $.
\end{theorem}

\begin{proof}
\textbf{Continuity.} For any $u,v\in H_{0}^{1}(\Omega )$,%
\begin{equation*}
\left\vert \int_{\Omega }\nabla u\cdot \nabla v\right\vert \leq \Vert \nabla
u\Vert _{L^{2}(\Omega )}\,\Vert \nabla v\Vert _{L^{2}(\Omega )}.
\end{equation*}%
For the nonlocal term, by Proposition~\ref{prop:nonlocal},%
\begin{equation*}
\lambda \left\vert \int_{\Omega }G(u)\,v\right\vert \leq \lambda \,\Vert
G(u)\Vert _{L^{2}(\Omega )}\,\Vert v\Vert _{L^{2}(\Omega )}\leq \lambda
M\,\Vert u\Vert _{L^{2}(\Omega )}\,\Vert v\Vert _{L^{2}(\Omega )}.
\end{equation*}%
By the Poincar\'{e} inequality, $\Vert w\Vert _{L^{2}(\Omega )}\leq
C_{P}\Vert \nabla w\Vert _{L^{2}(\Omega )}$ for all $w\in H_{0}^{1}(\Omega )$%
, so there exists $C>0$ (depending on $\lambda $, $M$, $C_{P}$) such that%
\begin{equation*}
|a(u,v)|\leq C\,\Vert u\Vert _{H^{1}(\Omega )}\,\Vert v\Vert _{H^{1}(\Omega
)}.
\end{equation*}%
Thus $a$ is continuous. \medskip

\textbf{Coercivity.} Since $G(u)\geq 0$ a.e.,%
\begin{equation*}
a(u,u)=\Vert \nabla u\Vert _{L^{2}(\Omega )}^{2}+\lambda \int_{\Omega
}G(u)\,u\,dx\,dy\geq \Vert \nabla u\Vert _{L^{2}(\Omega )}^{2}.
\end{equation*}%
By Poincar\'{e}, $\Vert \nabla u\Vert _{L^{2}(\Omega )}^{2}\geq \alpha
\,\Vert u\Vert _{H^{1}(\Omega )}^{2}$ for some $\alpha >0$, hence $%
a(u,u)\geq \alpha \Vert u\Vert _{H^{1}(\Omega )}^{2}$. \medskip

\textbf{Existence and Uniqueness.} By Lax--Milgram, there exists a unique $%
u\in H_{0}^{1}(\Omega )$ solving $a(u,v)=L(v)$ for all $v\in
H_{0}^{1}(\Omega )$. \medskip

\textbf{Positivity via Proposition \ref{wmp}. }Since $f\geq 0$ almost
everywhere in $\Omega $, the Nonlocal Positivity Principle (Proposition \ref%
{wmp}) implies that $u\left( z\right) \geq 0$ for a.e. $z\in \Omega $.

\textbf{Solution localization by means of sub- and supersolutions.}

\begin{itemize}
\item \emph{Subsolution:} $\underline{u} \equiv 0$ satisfies 
\begin{equation*}
-\Delta \underline{u} + \lambda G(\underline{u}) = 0 \leq f \quad \text{in }
\Omega,
\end{equation*}
with $\underline{u} = 0$ on $\partial\Omega$.

\item \emph{Supersolution:} Let $w$ solve%
\begin{equation*}
\begin{cases}
-\Delta w=f & \text{in }\Omega , \\ 
w=0 & \text{on }\partial \Omega .%
\end{cases}%
\end{equation*}%
By the classical maximum principle, $w\geq 0$ in $\Omega $. Since $G(w)\geq
0 $,%
\begin{equation*}
-\Delta w=f\geq f-\lambda G(w)\quad \text{in }\Omega ,
\end{equation*}%
so $w$ is a supersolution.
\end{itemize}

From the monotone iteration scheme developed in the part two of the paper
(see also \cite{AlvesCovei2015}),  one obtains the existence of a weak
solution $u$ satisfying $0\leq u$ a.e.\ in $\Omega $. graphical
representation of the subsolution, supersolution, and the resulting
numerical construction of the solution shows that $0\leq u\leq w$ a.e.\ in $%
\Omega $. Moreover, the assumptions on $f$ stated in the preamble of the
paper yield $u\not\equiv 0$ and $u>0$ in $\overline{\Omega }\setminus \Omega
_{0}$. 
\end{proof}

\medskip \textbf{Numerical Approximation Remark.} In the following sections, the weak form \eqref{eq:bilinear_form} is approximated using finite differences for the Laplacian and an integral quadrature or discrete convolution for $G(u)$. This ensures that the numerical bilinear form preserves the symmetry and positivity properties established above.

\subsection{PDE Formulation and Variational Consistency}

Define the bilinear form 
\begin{equation}
B(u,v):=\int_{\Omega }\int_{\Omega }K(z,\theta )\,u(\theta )\,v(z)\,d\theta
\,dz,\qquad u,v\in L^{2}(\Omega ).  \label{eq:Bdef}
\end{equation}%
Under the aforementioned assumptions, $G$ is bounded on $L^{2}(\Omega )$ and
satisfy 
\begin{equation*}
|B(u,v)|\leq M\,\Vert u\Vert _{L^{2}(\Omega )}\,\Vert v\Vert _{L^{2}(\Omega
)}
\end{equation*}%
by Schur's test; thus, all subsequent expressions are well-defined for $%
u,v\in L^{2}(\Omega )$. We consider the energy functional $J:
H_{0}^{1}(\Omega) \to \mathbb{R}$ given by 
\begin{equation}
J(u) = \frac{1}{2} \int_{\Omega} |\nabla u(z)|^{2} \, dz \;+\; \frac{\lambda%
}{2} \, B(u,u) \;-\; \int_{\Omega} f(z) \, u(z) \, dz,  \label{eq:Jdef}
\end{equation}
where $\lambda > 0$ and $f \in L^{2}(\Omega)$.

\begin{proposition}[First variation and Euler--Lagrange equation]
Let $J$ be defined as in \eqref{eq:Jdef}, with $B$ and $G$ as in %
\eqref{eq:Bdef} and \eqref{eq:G_def}, respectively. Then $J$ is Fr\'{e}chet
differentiable on $H_{0}^{1}(\Omega)$, and its first variation at $u$ in the
direction $v\in H_{0}^{1}(\Omega)$ is 
\begin{equation}
J^{\prime }(u)[v]\;=\;\int_{\Omega }\nabla u\cdot \nabla v\,dz\;+\;\lambda
\int_{\Omega }v(z)\left( \int_{\Omega }K_{s}(z,\theta )\,u(\theta )\,d\theta
\right) dz\;-\;\int_{\Omega }f\,v\,dz,  \label{eq:Jprime_general}
\end{equation}%
where%
\begin{equation*}
K_{s}(z,\theta ):=\tfrac{1}{2}\big(K(z,\theta )+K(\theta ,z)\big)
\end{equation*}
denotes the symmetrized kernel. Hence, 
\begin{equation}
\frac{\delta J}{\delta u}(z)\;=\;-\Delta u(z)\;+\;\lambda \int_{\Omega
}K_{s}(z,\theta )\,u(\theta )\,d\theta \;-\;f(z)\quad \text{in }%
H^{-1}(\Omega ).  \label{eq:var_derivative_general}
\end{equation}%
Specifically, if $K$ is symmetric a.e.\ on $\Omega \times \Omega $, i.e., $%
K(z,\theta )=K(\theta ,z)$, then $K_{s}=K$ and 
\begin{equation}
\frac{\delta J}{\delta u}(z)\;=\;-\Delta u(z)\;+\;\lambda \,G(u)(z)\;-\;f(z),
\label{eq:var_derivative_symmetric}
\end{equation}%
such that the Euler--Lagrange equation $\delta J/\delta u=0$ corresponds
exactly to the weak formulation of 
\begin{equation*}
-\Delta u+\lambda G(u)=f\quad \text{in }\Omega ,\qquad u|_{\partial \Omega
}=0.
\end{equation*}
\end{proposition}

\begin{proof}
Linearity and boundedness of $B(\cdot ,\cdot )$ follow from Schur's test,
hence $u\mapsto B(u,u)$ is well-defined and continuous on $L^{2}(\Omega )$.
Define $F(u):=\tfrac{1}{2}B(u,u)$. For $u,v\in L^{2}(\Omega )$ and $t\in 
\mathbb{R}$,%
\begin{equation*}
F(u+tv)=\frac{1}{2}B(u+tv,u+tv)=\frac{1}{2}\big(B(u,u)+2tB(u,v)+t^{2}B(v,v)%
\big).
\end{equation*}
Thus the G\^{a}teaux derivative of $F$ at $u$ in the direction $v$ is%
\begin{equation*}
F^{\prime }(u)[v]=\lim_{t\rightarrow 0}\frac{F(u+tv)-F(u)}{t}=B(u,v).
\end{equation*}
By polarization,%
\begin{equation*}
B(u,v)=\frac{1}{2}\big(B(u,v)+B(v,u)\big)=\int_{\Omega }v(z)\left(
\int_{\Omega }K_{s}(z,\theta )\,u(\theta )\,d\theta \right) dz,
\end{equation*}
with $K_{s}:=\tfrac{1}{2}(K+K^{\top })$. Therefore, 
\begin{equation}
F^{\prime }(u)[v]\;=\;\int_{\Omega }v(z)\left( \int_{\Omega }K_{s}(z,\theta
)\,u(\theta )\,d\theta \right) dz.  \label{eq:Fprime}
\end{equation}%
The gradient term satisfies%
\begin{equation*}
\frac{d}{dt}\bigg|_{t=0}\frac{1}{2}\int_{\Omega }|\nabla
(u+tv)|^{2}dz=\int_{\Omega }\nabla u\cdot \nabla v\,dz,
\end{equation*}
and the fidelity term is linear:%
\begin{equation*}
\frac{d}{dt}\bigg|_{t=0}\left( -\int_{\Omega }f(u+tv)\,dz\right)
=-\int_{\Omega }f\,v\,dz.
\end{equation*}
Combining these yields \eqref{eq:Jprime_general}. Interpreting $J^{\prime
}(u)[\cdot]$ as an element of $H^{-1}(\Omega)$ via the Riesz representation theorem yields the distributional derivative \eqref{eq:var_derivative_general}.
\end{proof}
\paragraph{Extension and Boundary Consistency.}

The numerical evaluation of nonlocal terms and differential operators near the boundary of $\Omega$ requires defining the function $u$ on an extended domain $\widetilde{\Omega}$. The next section formalizes the concept of extension operators, which allow for a consistent treatment of boundary effects in the discrete scheme.

\section{Extension and Padding Operators\label{sec:padding}}

Let $\Omega \subset \mathbb{R}^n$ be the domain of interest and let $\widetilde{%
\Omega} \supset \Omega$ be an \emph{extension domain}. Given $%
u:\Omega \to \mathbb{R}$, an extension (or padding) strategy is a linear operator
\begin{equation*}
E:L^{2}(\Omega )\longrightarrow L^{2}(\widetilde{\Omega }),\qquad u\mapsto 
\widetilde{u}:=E(u),
\end{equation*}%
that defines values of $u$ outside $\Omega $. Denote by $R:L^{2}(\widetilde{%
\Omega })\rightarrow L^{2}(\Omega )$ the restriction operator, $%
(Rw)(z):=w(z) $ for $z\in \Omega $. Given a kernel $K: \widetilde{\Omega}
\times \widetilde{\Omega} \to \mathbb{R}$, define the nonlocal operator on $%
\widetilde{\Omega}$ by 
\begin{equation*}
\widetilde{G}(w)(z) \;=\; \int_{\widetilde{\Omega}} K(z,\theta) \, w(\theta)
\, d\theta, \qquad z \in \widetilde{\Omega}.
\end{equation*}
The \emph{padded} nonlocal operator acting on $u$ over $\Omega$ is 
\begin{equation}
G_{\mathrm{pad}}(u) \;:=\; R\big( \widetilde{G}(E(u)) \big), \qquad G_{%
\mathrm{pad}}(u)(z) \;=\; \int_{\widetilde{\Omega}} K(z,\theta) \, 
\widetilde{u}(\theta) \, d\theta, \quad z \in \Omega.  \label{eq:Gpad}
\end{equation}
The PDE then reads (weakly on $\Omega$) 
\begin{equation}
-\Delta u(z) \;+\; \lambda \, G_{\mathrm{pad}}(u)(z) \;=\; f(z), \qquad z
\in \Omega,  \label{eq:PDEpad}
\end{equation}
with homogeneous Dirichlet trace $u|_{\partial\Omega} = 0$ unless stated
otherwise.

\subsection{Canonical padding operators on rectangular domains}

Assume $\Omega = (0,1)^2$ and $\widetilde{\Omega} = \mathbb{R}^2$. For $x =
(x_1,x_2) \in \mathbb{R}^2$ define:

\begin{itemize}
\item \textbf{Zero (Dirichlet) padding:}%
\begin{equation*}
E_{0}(u)(x)\;=\;\mathbf{1}_{\Omega }(x)\,u(x)\;=\;%
\begin{cases}
u(x), & x\in \Omega , \\ 
0, & x\notin \Omega .%
\end{cases}%
\end{equation*}

\item \textbf{Replicate (clamped) padding:} Let%
\begin{equation*}
c(t):=\min \{1,\max \{0,t\}\},\quad c(x):=(c(x_{1}),c(x_{2})),
\end{equation*}
and set $E_{\mathrm{rep}}(u)(x):=u(c(x))$.

\item \textbf{Reflect (even) padding:} Let%
\begin{equation*}
p(t):=t-2\lfloor t/2\rfloor \in \lbrack 0,2),\quad r(t):=%
\begin{cases}
p(t), & p(t)\in \lbrack 0,1], \\ 
2-p(t), & p(t)\in (1,2),%
\end{cases}%
\end{equation*}
with $r(x):=(r(x_{1}),r(x_{2}))$. Set $E_{\mathrm{ref}}(u)(x):=u(r(x))$.

\item \textbf{Periodic (wrap) padding:} Let%
\begin{equation*}
q(t):=t-\lfloor t\rfloor \in \lbrack 0,1),\quad q(x):=(q(x_{1}),q(x_{2})),
\end{equation*}
and set $E_{\mathrm{per}}(u)(x):=u(q(x))$.
\end{itemize}

\subsection{Boundary-condition interpretation}

For isotropic kernels $K(z,\theta) = \kappa(|z-\theta|)$:

\begin{itemize}
\item $E_{0}$ corresponds to homogeneous Dirichlet behavior for the nonlocal
term.

\item $E_{\mathrm{ref}}$ corresponds to homogeneous Neumann behavior (even
reflection).

\item $E_{\mathrm{per}}$ yields periodic boundary conditions.
\end{itemize}

Thus, the choice of $E$ fixes the effective boundary behavior of both the
nonlocal term and its discrete implementation.

\subsection{Variational formulation with padding}

Assume $K$ satisfies Schur-type bounds on $\widetilde{\Omega} \times 
\widetilde{\Omega}$ ensuring boundedness of $\widetilde{G}: L^{2}(\widetilde{%
\Omega}) \to L^{2}(\widetilde{\Omega})$. Define the padded energy 
\begin{equation*}
J_{\mathrm{pad}}(u) \;=\; \frac12 \int_{\Omega} |\nabla u|^{2} \, dz \;+\; 
\frac{\lambda}{2} \int_{\Omega} \int_{\widetilde{\Omega}} K(z,\theta) \,
E(u)(\theta) \, u(z) \, d\theta \, dz \;-\; \int_{\Omega} f \, u \, dz.
\end{equation*}
Then $J_{\mathrm{pad}}$ is G\^ateaux differentiable on $H_{0}^{1}(\Omega)$
and its first variation satisfies 
\begin{equation*}
\frac{\delta J_{\mathrm{pad}}}{\delta u}(z)\;=\;-\Delta u(z)\;+\;\lambda
\,G_{\mathrm{pad}}(u)(z)\;-\;f(z)\quad \text{in }H^{-1}(\Omega ),
\end{equation*}%
so the Euler--Lagrange equation $\delta J_{\mathrm{pad}}/\delta u=0$ is
exactly \eqref{eq:PDEpad}. If $K$ is symmetric a.e.\ on $\widetilde{\Omega }%
\times \widetilde{\Omega }$, then $G_{\mathrm{pad}}$ is self-adjoint on $%
L^{2}(\Omega )$ as $R\widetilde{G}E$ with $E,R$ bounded.

\section{Numerical Discretization and Fixed-Point Analysis\label{III}}

In two dimensions, the Laplacian is given by 
\begin{equation*}
\Delta u(x,y)=\frac{\partial ^{2}u}{\partial x^{2}}(x,y)+\frac{\partial ^{2}u%
}{\partial y^{2}}(x,y),
\end{equation*}%
and on a uniform grid ($h=1$), we discretize using central differences: 
\begin{equation}
\Delta u_{i,j}\;\approx \;u_{i+1,j}+u_{i-1,j}+u_{i,j+1}+u_{i,j-1}-4\,u_{i,j}.
\label{dis1}
\end{equation}
We solve the elliptic problem 
\begin{equation*}
-\Delta u+\lambda \,G(u)=f
\end{equation*}%
by the fixed-point iteration 
\begin{equation}
u^{(n+1)}=u^{(n)}+\tau \Bigl(f+\Delta u^{(n)}-\lambda \,G(u^{(n)})\Bigr),
\label{eq:iteration1}
\end{equation}%
enforcing $u^{(n)}|_{\partial \Omega }=0$ at each step. Equivalently, define 
\begin{equation}
T(u)\;=\;u+\tau \Bigl(f+\Delta u-\lambda \,G(u)\Bigr),\quad \tau >0,
\label{fo2}
\end{equation}%
so that iterates satisfy $u^{(n+1)}=T(u^{(n)})$. Convergence follows if $T$
is a contraction in a suitable norm.

\begin{theorem}[Contraction of the Fixed-Point Iteration]
\label{cont} Suppose the hypotheses on $G$ and $f$ from Theorem \ref%
{thm:existence} and let $T$ be as in \eqref{fo2}. In addition, assume $G$ is
Lipschitz with constant $L_{G}$ and choose $\tau >0$ so that%
\begin{equation*}
0<\tau <\frac{2}{(\Vert \Delta \Vert +\lambda L_{G})^{2}}\left( \frac{1}{%
C_{P}^{2}}-\lambda L_{G}\right) ,\qquad 
\end{equation*}%
where the regularization parameter $\lambda >0$ satisfies%
\begin{equation*}
\lambda <\frac{1}{L_{G}C_{P}^{2}},
\end{equation*}%
and $C_{P}$ is the Poincar\'{e} constant on $\Omega $. Then $T$ is a
contraction in the $L^{2}\left( \Omega \right) $ (or $H^{1}\left( \Omega
\right) $) norm, and the sequence $\{u^{(n)}\}$ in \eqref{eq:iteration1}
converges to the unique weak positive solution of%
\begin{equation*}
-\Delta u+\lambda \,G(u)=f\quad \text{in }\Omega ,\quad u=0\ \text{on }%
\partial \Omega .
\end{equation*}
\end{theorem}

\begin{proof}
\textbf{Step 1: Error propagation.} Let $w^{(n)}=u^{(n)}-u$, where $u$ is
the exact solution. Subtracting the fixed-point update from $u$ yields 
\begin{equation}
w^{(n+1)}=w^{(n)}+\tau \left[ \Delta w^{(n)}-\lambda \big(G(u^{(n)})-G(u) %
\big)\right].  \label{eq:error}
\end{equation}

\textbf{Step 2: Energy estimate.} Taking the $L^{2}$-norm squared in %
\eqref{eq:error} gives 
\begin{align}
\Vert w^{(n+1)}\Vert ^{2}& =\Vert w^{(n)}\Vert ^{2}+2\tau \langle
w^{(n)},\Delta w^{(n)}-\lambda (G(u^{(n)})-G(u))\rangle \\
& \quad +\tau ^{2}\Vert \Delta w^{(n)}-\lambda (G(u^{(n)})-G(u))\Vert ^{2}.
\label{eq:energy}
\end{align}

\textbf{Step 3: Term-by-term bounds.}

\begin{itemize}
\item[(a)] \emph{Laplacian term:} By integration by parts and Poincar\'{e},%
\begin{equation*}
\langle w^{(n)},\Delta w^{(n)}\rangle =-\Vert \nabla w^{(n)}\Vert ^{2}\leq -%
\frac{1}{C_{P}^{2}}\Vert w^{(n)}\Vert ^{2}.
\end{equation*}

\item[(b)] \emph{Nonlocal term:} Lipschitz continuity of $G$ gives%
\begin{equation*}
\Vert G(u^{(n)})-G(u)\Vert \leq L_{G}\Vert w^{(n)}\Vert ,
\end{equation*}
so 
\begin{equation*}
2\tau \lambda \langle w^{(n)},-(G(u^{(n)})-G(u))\rangle \leq 2\tau \lambda
L_{G}\Vert w^{(n)}\Vert ^{2}.
\end{equation*}

\item[(c)] \emph{Second-order term:} With $L=\Vert \Delta \Vert +\lambda
L_{G}$,%
\begin{equation*}
\Vert \Delta w^{(n)}-\lambda (G(u^{(n)})-G(u))\Vert \leq L\Vert w^{(n)}\Vert
,
\end{equation*}%
so the last term in \eqref{eq:energy} is bounded by $\tau ^{2}L^{2}\Vert
w^{(n)}\Vert ^{2}$.
\end{itemize}

\textbf{Step 4: Contractivity.} Combining the above,%
\begin{equation*}
\Vert w^{(n+1)}\Vert ^{2}\leq \left[ 1-\frac{2\tau }{C_{P}^{2}}+2\tau
\lambda L_{G}+\tau ^{2}L^{2}\right] \Vert w^{(n)}\Vert ^{2}.
\end{equation*}%
Requiring the bracket to be $<1$ yields the stated condition on $\tau $ and
the necessary assumption 
\begin{equation}
\frac{1}{C_{P}^{2}}>\lambda L_{G}.  \label{nc}
\end{equation}%
Under these, $T$ is a contraction and Banach's fixed-point theorem gives
convergence to the unique weak solution.
\end{proof}

\begin{remark}
The condition \eqref{nc} links the geometry of $\Omega$ (via $C_{P}$) with
the parameter $\lambda$ and the operator $G$. In practice,
one chooses $\tau$ to satisfy this stability inequality.
\end{remark}

\subsection{Alternative View: Explicit Euler for a Gradient Flow}

The fixed-point iteration \eqref{eq:iteration1} can also be viewed as an
explicit Euler discretization of the gradient descent flow associated with
the energy functional $J(u)$. In this interpretation, 
\begin{equation*}
u^{(n+1)}=u^{(n)}-\tau \,\nabla J(u^{(n)})
\end{equation*}%
requires $\tau $ to be sufficiently small relative to the Lipschitz constant
of $\nabla J$. The stability condition derived above ensures numerical
stability in this gradient-flow perspective as well.

\subsection{Numerical Experiments and Verification}

We present several numerical experiments to validate the theoretical findings. We consider the domain $\Omega = (0,1)^2$ and solve the nonlocal elliptic equation \eqref{eq:PDE} using the iterative scheme described in Theorem~\ref{cont}.

\begin{figure}[H]
\centering
\subfloat{\includegraphics[width=0.6\textwidth]{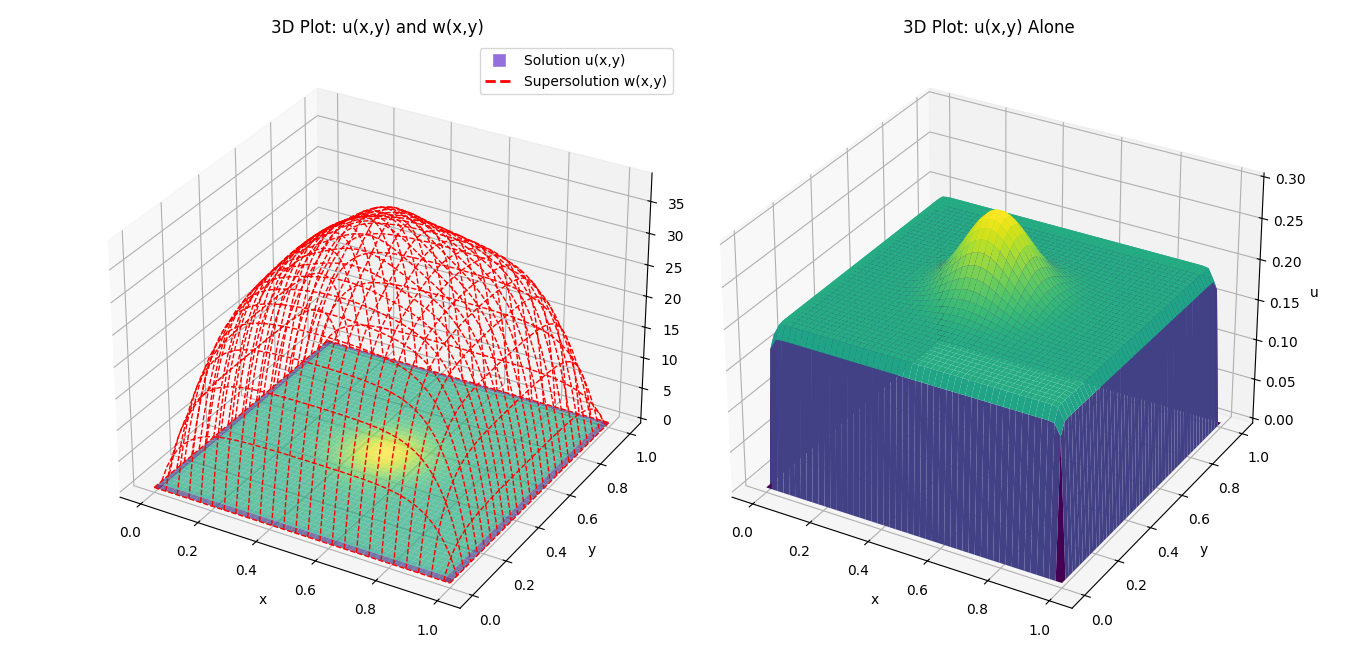}}
\caption{Three-dimensional visualization of the numerical solution $u(x,y)$
(solid surface) and the corresponding supersolution $w(x,y)$ (wireframe) over the domain $%
\Omega = [0,1] \times [0,1]$.}
\label{fig:3Dsurface}
\end{figure}

The smooth profile of the solution $u$ reflects the combined regularization of the Laplacian and the nonlocal Gaussian kernel. The close correspondence between the solution and the constructed supersolution confirms the validity of the monotone iteration approach.

\begin{figure}[H]
\centering
\includegraphics[width=0.7\textwidth]{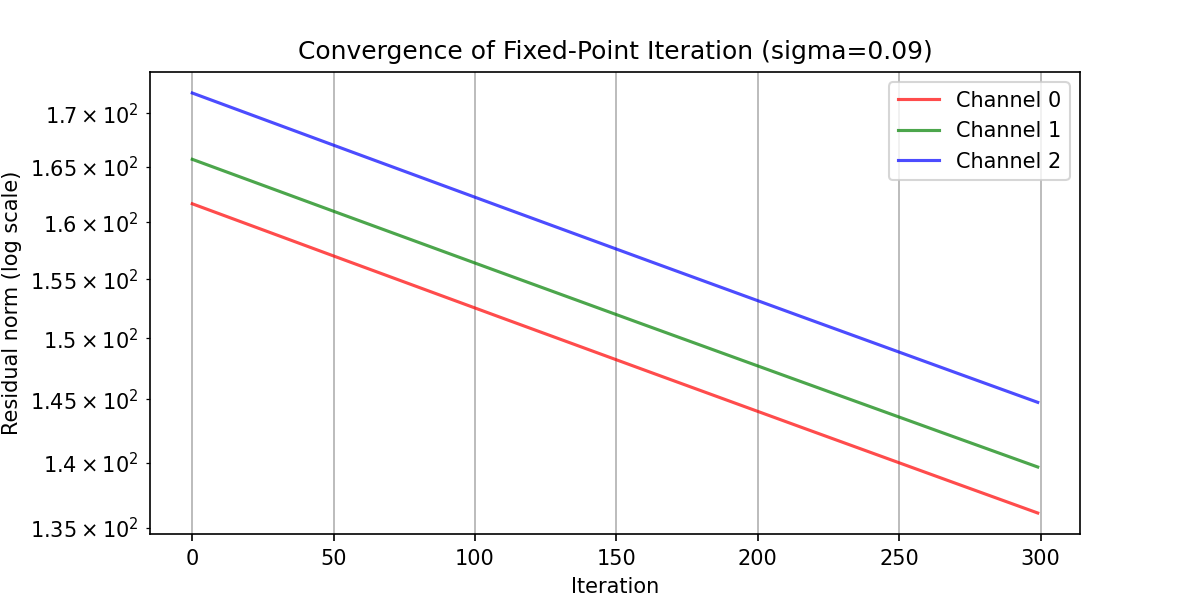}
\caption{Convergence of the fixed-point iteration: residual norm $\|r^{(k)}\|$ versus iteration number for various parameter configurations. The monotonic decay validates the contraction conditions established in Theorem~\ref{cont}.}
\label{fig:convergence}
\end{figure}

\section{Conclusion and Future Work\label{IV}}

This study presented a comprehensive analysis of a class of nonlocal elliptic equations, combining a rigorous mathematical framework with a practical numerical scheme. The formulation accommodates a general nonlocal operator $G$ defined by convolution with a kernel $K$. Under standard integrability and boundedness conditions on $K$, we established that $G$ is a bounded operator on $L^2(\Omega)$. These analytical findings, combined with classical energy estimates for the Laplacian, allow the application of the Lax--Milgram theorem to guarantee the existence and uniqueness of the weak solution in $H_0^1(\Omega)$.

The numerical discretization utilized finite differences for the local operator and discrete convolution for the nonlocal term. By incorporating extension operators, we ensured consistency near the boundaries, preventing numerical artifacts that often arise in nonlocal problems on bounded domains. The stability of the fixed-point iterative scheme was rigorously analyzed, providing explicit conditions on the step size $\tau$ and the regularization parameter $\lambda$ based on the spectral properties of the operators.

\medskip \noindent \textbf{Summary of contributions:}
\begin{itemize}
\item Established a rigorous mathematical foundation for the nonlocal elliptic model.
\item Proved well-posedness (existence, uniqueness, and positivity) via variational methods.
\item Formalized the role of padding as extension operators for nonlocal consistency.
\item Derived explicit stability and convergence conditions for the iterative solver.
\item Validated the theory through numerical experiments on bounded domains.
\end{itemize}

\medskip \noindent \textbf{Future directions:}
\begin{itemize}
\item \emph{Non-symmetric kernels:} Extend the analysis to kernels where $K(z,\theta) \neq K(\theta,z)$, exploring the resulting loss of self-adjointness.
\item \emph{Adaptive regularization:} Investigate spatially-varying parameters $\lambda(z)$ to handle non-homogeneous source terms.
\item \emph{High-order schemes:} Explore spectral methods or higher-order finite elements for improved accuracy in the nonlocal term approximation.
\item \emph{Multidimensional applications:} Apply the framework to higher-dimensional problems in nonlocal mechanics or population dynamics.
\end{itemize}

In conclusion, the proposed nonlocal elliptic model strikes a balance between local diffusion and global interaction. By unifying rigorous analysis with a consistent numerical implementation, we have provided a versatile mathematical tool suitable for a wide range of nonlocal problems.

\section*{Declaration of Generative AI and AI-assisted technologies in the writing process}
The authors utilized the free AI tool Grammarly to improve the manuscript’s grammar. Following its use, the authors reviewed 
and edited the content as necessary, taking full responsibility for the final publication.
\section*{Declaration of competing interest}
This work does not have any conflicts of interest.
\section*{Funding}
This research did not receive any specific grant from funding agencies in the public, commercial, or not-for-profit sectors.
\section*{Data availability}
No data was used for the research described in the article.
\section*{Acknowledgments}
The author expresses his gratitude to the referees. Some of them contributed substantially to strengthening the theoretical results in the present version, clearly highlighting aspects that had escaped me in the initial form of the paper. Other referees offered valuable suggestions regarding the practical application of the model and its contextualization through comparisons with existing approaches. I thank them once again and assure them that my intention has been to advance the current line of research.

\appendix
\section{Appendix: Numerical Implementation}

Below we provide the Python implementation of the numerical scheme used to validate the mathematical results. The code computes the numerical solution $u(x,y)$ and compares it with the corresponding supersolution $w(x,y)$, ensuring that the theoretical bounds are respected.

\begin{lstlisting}[caption={Python Code for Nonlocal Elliptic PDE Solver}]
import numpy as np
import matplotlib.pyplot as plt
from scipy.ndimage import gaussian_filter
from mpl_toolkits.mplot3d import Axes3D
from matplotlib.lines import Line2D

def compute_laplacian(u):
    lap = np.zeros_like(u)
    lap[1:-1, 1:-1] = (u[2:, 1:-1] + u[:-2, 1:-1] +
                       u[1:-1, 2:] + u[1:-1, :-2] -
                       4 * u[1:-1, 1:-1])
    return lap

def enforce_boundary(u):
    u[0, :], u[-1, :], u[:, 0], u[:, -1] = 0, 0, 0, 0
    return u

def solve_pde(f, lam=5.0, tau=0.1, max_iter=300, tol=1e-4, sigma=0.5):
    u = f.copy()
    u = enforce_boundary(u)
    for it in range(max_iter):
        lap = compute_laplacian(u)
        G_u = gaussian_filter(u, sigma=sigma)
        residual = f + lap - lam * G_u
        u = u + tau * residual
        u = enforce_boundary(u)
        res_norm = np.linalg.norm(residual)
        if res_norm < tol:
            break
    return u

def solve_poisson(f, tau=0.1, max_iter=300, tol=1e-4):
    w = f.copy()
    w = enforce_boundary(w)
    for it in range(max_iter):
        lap = compute_laplacian(w)
        residual = f + lap
        w = w + tau * residual
        w = enforce_boundary(w)
        res_norm = np.linalg.norm(residual)
        if res_norm < tol:
            break
    return w

def main():
    nx, ny = 50, 50  
    x = np.linspace(0, 1, nx)
    y = np.linspace(0, 1, ny)
    X, Y = np.meshgrid(x, y)
    f = 1 + 0.5 * np.exp(-((X - 0.5)**2 + (Y - 0.5)**2) / 0.02)
    u = solve_pde(f, lam=5.0, tau=0.1, max_iter=300, tol=1e-4, sigma=0.5)
    w = solve_poisson(f, tau=0.1, max_iter=300, tol=1e-4)
    fig = plt.figure(figsize=(10, 6))
    ax = fig.add_subplot(111, projection='3d')
    ax.plot_surface(X, Y, u, cmap='viridis', alpha=0.7)
    ax.plot_wireframe(X, Y, w, color='red', linestyle='--', rstride=2, cstride=2)
    plt.show()

if __name__ == '__main__':
    main()
\end{lstlisting}


\begin{thebibliography}{99}

\bibitem{AlvesCovei2015}
C.O. Alves, D.-P. Covei,
Existence of solution for a class of nonlocal elliptic problem via sub-supersolution method,
Nonlinear Anal. Real World Appl. 23 (2015) 1--8.

\bibitem{Brezis2010}
H. Brezis,
Functional Analysis, Sobolev Spaces and Partial Differential Equations,
Springer, New York, 2010.

\bibitem{Evans1998}
L.C. Evans,
Partial Differential Equations,
Grad. Stud. Math., AMS, Providence, 1998.

\bibitem{GilboaOsher2008}
G. Gilboa, S. Osher,
Nonlocal operators with applications to image processing,
Multiscale Model. Simul. 7 (2008) 1005--1028.

\bibitem{Leung2009}
S. Leung, G. Liang, K. Solna, H. Zhao,
Expectation-maximization algorithm with local adaptivity,
SIAM J. Imaging Sci. 2 (2009) 834--857.

\bibitem{Menendez2024}
A.G. Men\'{e}ndez,
Physics meets pixels: PDE models in image processing,
arXiv preprint arXiv:2412.11946, 2024.

\end{thebibliography}
\end{document}